\documentclass[10pt,a4paper,leqno]{amsart}
\usepackage{amssymb, amsmath, amscd}
\usepackage[hypertexnames=false]{hyperref}

\usepackage{color}

\DeclareMathOperator{\Hes}{\rm Hes}

\DeclareMathOperator{\Ric}{\rm Ric}

\newcommand{\nh}{\nabla h}

\makeatletter
\@namedef{subjclassname@2020}{%
	\textup{2020} Mathematics Subject Classification}
\makeatother

\newtheorem{theorem}{Theorem}[section]
\newtheorem{lemma}[theorem]{Lemma}

\newtheorem{corollary}[theorem]{Corollary}

\theoremstyle{definition}

\theoremstyle{remark}
\newtheorem{remark}[theorem]{Remark}

\newcommand\restr[2]{{
  \left.\kern-\nulldelimiterspace 
  #1 
  \vphantom{\big|} 
  \right|_{#2} 
  }}

\begin{document}
\title[{The vacuum weighted Einstein field equations on pr-waves}]{The vacuum weighted Einstein field equations on pure radiation waves}
\author{M. Brozos-V\'azquez, D. Moj\'on-\'Alvarez}
\address{MBV: CITMAga, 15782 Santiago de Compostela, Spain}
\address{\phantom{MBV:}  
	Universidade da Coru\~na, Campus Industrial de Ferrol, Department of Mathematics, 15403 Ferrol, Spain}
\email{miguel.brozos.vazquez@udc.gal}
\address{DMA: CITMAga, 15782 Santiago de Compostela, Spain}
\address{\phantom{DMA:}
	University of Santiago de Compostela,
15782 Santiago de Compostela, Spain}
\email{diego.mojon@rai.usc.es}
\thanks{Research partially supported by grants PID2022-138988NB-I00 funded by MICIU/AEI/10.13039/501100011033 and by ERDF, EU, and ED431F 2020/04 (Xunta de Galicia, Spain); and by contract FPU21/01519 (Ministry of Universities, Spain).}
\date{}

\maketitle

\begin{abstract}
We classify solutions of the vacuum weighted Einstein field equations on smooth metric measure spacetimes $(M,g,h\, dvol_g)$ of dimension 4, where the underlying manifold $(M,g)$ is a $pr$-wave. We use this result to provide examples of solutions with some special geometric properties. 

The gradient $\nabla h$ is lightlike or spacelike. In the first case, the underlying manifold is a $pp$-wave. In the second case, the Ricci operator is nilpotent. Moreover, $2$-step nilpotent solutions are also realized on $pp$-waves.
\end{abstract}


\section{Introduction}

A usual spacetime can be generalized by introducing a positive density function $h$. This density alters the Riemannian volume element $dvol_g$ and gives rise to a {\it smooth metric measure spacetime (SMMS)}, i.e., a triple $(M,g,h\, dvol_g)$, or $(M,g,h)$ for short, where $(M,g)$ is a Lorentzian manifold and $h$ is a positive smooth function on $M$. We assume the density to be non-constant on any open subset of $M$ so that $\nh\neq 0$. The study of SMMSs is based around geometric objects known as {\it weighted invariants}, which reflect the influence of the density on the geometry of the underlying spacetime. 

One such object is the {\it weighted Einstein tensor},
\[
G^h=h\rho-\Hes_h+\Delta h g,
\]
where $\rho$ is the Ricci tensor and $\Hes_h$ and $\Delta h$ denote the Hessian tensor and the Laplacian of $h$, respectively. $G^h$ was introduced in \cite{Brozos-Mojon} as a tensor which is symmetric, divergence-free if the scalar curvature is constant, concomitant of the metric, the density function and their first two derivatives, and linear on the second derivatives of $g$ as well as the density function. These properties are reminiscent of those that characterize the usual Einstein tensor $G=\rho-\frac{\tau}2g$ of General Relativity (see \cite{Lovelock} for details). Moreover, it was shown in \cite{Brozos-Mojon-noniso} that a suitable modification of the variational problem of the Einstein-Hilbert functional, in order to include the new volume element $h\, dvol_g$, yields the condition
\begin{equation}\label{eq:vacuum-Einstein-field-equations}
		G^h=h\rho -\Hes_h+\Delta h g=0
\end{equation}
at critical points. These two facts suggest that this is a suitable weighted analogue to the vacuum Einstein field equations of General Relativity, and thus the differential equations given by \eqref{eq:vacuum-Einstein-field-equations} shall be referred to as the {\it vacuum weighted Einstein field equations}. Note that equation~\eqref{eq:vacuum-Einstein-field-equations} does not only arise naturally in the context of SMMSs, but also as a second order equation with several geometric interests. Indeed, it was previously considered in Riemannian signature, not only from a purely geometric standpoint, but also due to its relation with static spacetimes. The equation itself already appeared through the variation of the scalar curvature (see \cite{Bla,Marsden}), but later works by Kobayashi and Lafontaine \cite{Kobayashi, Lafontaine} laid the groundwork for a more systematic study in the Riemannian setting (see also \cite{Shen,Yuan}).

As in the usual setting, a {\it solution} of the vacuum weighted Einstein field equations will be regarded as a SMMS $(M,g,h)$ such that $G^h=0$ on $M$. Thus, for any solution, the divergence-free property of $G^h$, which is also characteristic of the usual Einstein tensor, is always satisfied. This gives a first geometric feature common to all solutions and also found in the standard vacuum solutions of General Relativity.

\begin{lemma}\label{le:const-sc} {\rm \cite{Brozos-Mojon}}
	If $(M,g,h)$ is a solution of the vacuum weighted Einstein field equations, then its scalar curvature $\tau$ is constant.
\end{lemma}

Many interesting spacetimes that appear in General Relativity are characterized by the presence of a distinguished vector field with some particular features. This is the case of $pr$-waves, which are a class of Brinkmann waves that generalizes $pp$-waves. We refer to Section~\ref{sect2} for more details on this kind of manifold. In this note, we search for solutions of \eqref{eq:vacuum-Einstein-field-equations} with this particular underlying structure and we see that different kinds of solutions arise. 

Solutions can be broadly classified according to the causal character of $\nh$, where those with $\nh$ lightlike are called {\it isotropic}, and those with $\nh$ timelike or spacelike are referred to as {\it non-isotropic}. The geometric characteristics of solutions, as well as the approach needed to analyze equation~\eqref{eq:vacuum-Einstein-field-equations}, are often quite different depending on whether they are isotropic or non-isotropic. The analysis of solutions on $pr$-waves that we perform in Section \ref{sect:pr-waves} gives evidence of this fact.

In summary, the main objective of this work is to further our understanding of the vacuum weighted Einstein field equations by studying $4$-dimensional solutions among $pr$-waves. The special relevance of this family lies in its geometric and physical properties, and in its relation with known results about solutions of \eqref{eq:vacuum-Einstein-field-equations}, as we discuss in the following sections. In Section~\ref{sect:examples} we analyze more specific examples, classifying $pr$-waves with harmonic curvature and providing geodesically complete examples on specific families.

\section{Pure radiation waves}\label{sect2}

{\it Brinkmann waves} constitute a remarkable family of spacetimes that includes other special Lorentzian manifolds as subclasses. They are characterized by the presence of a recurrent vector field $V$, i.e. $\nabla_XV=\omega(X)V$  for a $1$-form $\omega$ on $M$ and every $X\in  \mathfrak{X}(M)$.

Imposing specific conditions on the curvature of  Brinkmann waves results in some families of interest with more restricted geometry. Thus, following terminology in \cite{leistner}, {\it pure radiation waves}  ($pr$-waves for short), are Lorentzian manifolds $(M,g)$ which have a recurrent lightlike vector field $V$ such that the curvature tensor satisfies $R(V^\perp,V^\perp,-,-)=0$, where $V^\perp$ is the orthogonal complement of $V$. For a $4$-dimensional $pr$-wave, there exist local coordinates $(u,v,x,y)$ such that the metric takes the form
\begin{equation}\label{eq:local-coord-pr-xeral}
	g=2dudv+F(u,v,x,y)dv^2+dx^2+dy^2.
\end{equation}
Whenever the distinguished vector field $V$ is parallel, the $pr$-wave is said to be a $pp$-wave, and coordinates can be specialized so that the function  $F$ in \eqref{eq:Fprwave-1} satisfies $\partial_uF=0$. Note that a Brinkmann wave with parallel lightlike vector field $V$ is a $pp$-wave if and only if $R(V^\perp,V^\perp,\cdot,\cdot)=0$. Moreover, it was pointed out in \cite{leistner} that a $pr$-wave is a $pp$-wave if and only if it is Ricci isotropic, this is, the image of the Ricci operator $\Ric$ is totally isotropic.

Furthermore, a {\it plane wave} is a $pp$-wave with transversally parallel curvature tensor (i.e., such that $\nabla_{V^\perp} R=0$).  These families arise frequently  as solutions of the Einstein equations in special situations described in General Relativity, and we refer to \cite{stephani-et-al} for examples of contexts where they play a role. The metric of a plane wave takes the form \eqref{eq:local-coord-pr-xeral} where $F(u,v,x,y)=a(v)x^2+b(v)xy+c(v) y^2$ and the coefficients $a$, $b$, and $c$ are smooth functions of $v$. If $a$, $b$, and $c$ are constants, these metrics correspond to {\it Cahen-Wallach symmetric spaces} \cite{cahen-wallach} (see also \cite{Cahen-Leroy-Parker-Tricerri-Vanhecke}), a well-known family of geodesically complete Lorentzian manifolds.

Previous works on the vacuum weighted Einstein field equations (see \cite{Brozos-Mojon,Brozos-Mojon-noniso}) illustrate that spacetimes with a distinguished lightlike vector field play a central role among the set of solutions. Indeed, Brinkmann waves, and more broadly Kundt spacetimes (which we will not discuss in this note), are of special importance for isotropic solutions, but also arise in the non-isotropic case under some geometric conditions. For example, it was shown in \cite{Brozos-Mojon} that isotropic solutions are necessarily Kundt spacetimes and, moreover, that they are Brinkmann waves if the Ricci operator is $2$-step nilpotent. Also, in dimension four, isotropic Ricci-flat solutions are $pp$-waves. In the non-isotropic case, if the curvature is harmonic, it is shown in \cite{Brozos-Mojon-noniso} that all solutions with non-diagonalizable Ricci operator are realized on Kundt spacetimes.

\section{Vacuum solutions realized on pr-waves}\label{sect:pr-waves}

Motivated by the relation between vacuum weighted solutions and spacetimes with a parallel line field pointed out in Section~\ref{sect2}, we consider the family of $pr$-waves.  In this section we carry out a systematic search of solutions in this context and provide a complete classification as follows.


\begin{theorem}\label{th:pr-waves-classification}
	Let $(M,g,h)$ be a 4-dimensional, non-flat solution of the vacuum weighted Einstein field equations~\eqref{eq:vacuum-Einstein-field-equations} realized on a $pr$-wave. Then $\nh$ is lightlike or spacelike, and the following holds:
	\begin{enumerate}
		\item If $\nh$ is lightlike, then $(M,g)$ is a $pp$-wave with harmonic curvature. Moreover, there exist local coordinates as in \eqref{eq:local-coord-pr-xeral} with $\partial_x^2 F+\partial_y^2 F=\gamma(v)$ and $\partial_uF=0$ such that the density function $h=h(v)$ satisfies $2h''+h\gamma=0$.
		\item If $\nh$ is spacelike, then $\Ric$ is nilpotent and, moreover:
		\begin{enumerate}
			\item If $\Ric$ is $2$-step nilpotent, then $(M,g)$ is a $pp$-wave. Moreover, there exist coordinates as in \eqref{eq:local-coord-pr-xeral}
			 with $\partial_uF=0$ and density function of the form $h(v,x,y)=h_0(v)+(x+Ay)h_x$, with $A\in\mathbb{R}$ and $h_x\neq 0$, satisfying 
			\begin{equation}\label{eq:condition-pp-wave}
			 0=-2 G^h_{vv}=	2h_0''+h_x(\partial_xF+A \partial_yF)+h(\partial_x^2F+\partial_y^2F).
			\end{equation}
			\item If $\Ric$ is $3$-step nilpotent, then there exist $(u,v,x,y)$ such that the density function takes the form $h(v,x,y)=h_0(v)+(x+Ay)h_x(v)$, with $A\in\mathbb{R}$ and $h'_x\neq 0$, and the metric is as in \eqref{eq:local-coord-pr-xeral} with 
			\[
			F(u,v,x,y)=F_0(v,x,y)+u \left(\frac{2 h_x'(v)
				\log (h(v,x,y))}{h_x(v)}+\alpha (v)\right)
			\]
			satisfying
			\begin{equation}\label{eq:remaining-condition-pr-wave}
				\begin{array}{rcl}
					0=-2h_x G^h_{vv}&=&2h_x'(h_0'+(x+Ay)h_x')\log(h_0+(x+Ay)h_x) \\
					\noalign{\medskip}
					&&+h_x^2(\partial_xF_0+A\partial_yF_0+(x+Ay)(\partial_x^2F_0+\partial_y^2F_0)) \\
					\noalign{\medskip}
					&&+h_x(\alpha(h_0'+(x+Ay)h_x')+2h_0''+2(x+Ay)h_x'') \\
					\noalign{\medskip}
					&&+h_xh_0(\partial_x^2F_0+\partial_y^2F_0).
				\end{array}
			\end{equation}
		\end{enumerate}
	\end{enumerate}
\end{theorem}

We begin our arguments to provide the proof of Theorem~\ref{th:pr-waves-classification} with the following lemma, which gives the general form of the metric and the density function of every possible solution.

\begin{lemma}\label{lemma:$pr$-wave-1}
	Let $(M,g,h)$ be a 4-dimensional non-flat solution of the vacuum weighted Einstein field equations, realized on a $pr$-wave with metric given by \eqref{eq:local-coord-pr-xeral} in the coordinates $(u,v,x,y)$. Then, the Ricci operator is nilpotent, and the density function $h$ takes the form
	\begin{equation}\label{eq:hprwave-1}
		h(v,x,y)=h_x(v)x+h_y(v)y+h_0(v),
	\end{equation}
	while the function $F$ that defines the $pr$-wave is given by
	\begin{equation}\label{eq:Fprwave-1}
		F(u,v,x,y)=F_1(v,x,y)u+F_0(v,x,y),
	\end{equation}
	for suitable smooth functions $h_x$, $h_y$, $F_1$ and $F_0$.
\end{lemma}
\begin{proof}
Consider coordinates $(u,v,x,y)$ so that the metric of the $pr$-wave is written as in \eqref{eq:local-coord-pr-xeral}. Then, the scalar curvature of the manifold is given by $\tau=\partial_u^2F$. Since, by Lemma~\ref{le:const-sc}, $\tau$ is constant for any solution, it follows that $F$ takes the form $F(u,v,x,y)=\frac{\tau}{2}u^2+F_1(v,x,y)u+F_0(v,x,y)$. We simplify notation and denote $G^h(\partial_u,\partial_u)=G^h_{uu}$, $G^h(\partial_u,\partial_x)=G^h_{ux}$, $G^h(\partial_x,\partial_y)=G^h_{xy}$ and so on to compute the following components of the weighted Einstein tensor:
	\[
	G^h_{uu}=-\partial_u^2h, \quad G^h_{ux}=-\partial_u\partial_xh, \quad G^h_{uy}=-\partial_u\partial_yh, \quad G^h_{xy}=-\partial_x\partial_yh.
	\]
	Hence $h$ can be written as $h(u,v,x,y)=h_1(v)u+h_x(v,x)+h_y(v,y)$ for certain $h_1$, $h_x$ and $h_y$. Now, take the component $G^h_{yy}=-h_1(\tau u+F_1)+2 h_1'+\partial_x^2h_{x}$. Differentiate with respect to $u$ to find $0=\partial_uG^h_{yy}=-\tau h_1$. On the other hand, we compute the component
	\[
	2G^h_{vx}=-2\partial_v \partial_x h_x+ (h_x+h_y+2u h_1)\partial_x F_1+h_1 \partial_xF_0,
	\]
	hence $0=\partial_u G^h_{vx}=h_1 \partial_x F_1$. Similarly, we find $0=\partial_u G^h_{vy}=h_1 \partial_y F_1$. Thus, there are two possibilities: either  $h_1=0$ or $h_1\neq0$, $\tau=0$ and $F_1=F_1(v)$. We analyze them separately.
	\subsubsection*{Case 1: $h_1\neq0$, $\tau=0$ and $F_1=F_1(v)$} We will see that this case results in a flat manifold. Let $h_1\neq0$, $\tau=0$ and $F_1=F_1(v)$. The only non-vanishing components (up to symmetries) of the curvature tensor for a $pr$-wave with $F(u,v,x,y)=F_1(v)u+F_0(v,x,y)$ are
	\[
	\begin{array}{rcl}
		2R(\partial_x,\partial_v,\partial_v,\partial_x)&=&\partial_x^2F_0, \quad 2R(\partial_y,\partial_v,\partial_v,\partial_y)=\partial_y^2F_0, \\
		\noalign{\medskip}
		2R(\partial_y,\partial_v,\partial_v,\partial_x)&=&\partial_x\partial_y F_0.
	\end{array}
	\]
	Moreover, the weighted Einstein field equations  yield $0=2\partial_xG^h_{vy}=h_1\partial_x\partial_y F_0$. Since $h_1\neq 0$, we have $R(\partial_y,\partial_v,\partial_v,\partial_x)=0$. Additionally, we compute
	\[
	\begin{array}{rcl}
		0&=&G_{xx}^h=-h_1F_1+2h_1'+\partial_y^2h_y,\\
		\noalign{\medskip}
		0&=&G_{yy}^h=-h_1F_1+2h_1'+\partial_x^2h_x,\\
		\noalign{\medskip}
		0&=&2G^h_{uv}=- h_1 F_1+2h_1'+2\partial_x^2 h_x+2\partial_y^2 h_y.
	\end{array}
	\]
	On the one hand, $0=G_{xx}^h-G_{yy}^h=\partial_y^2h_y-\partial_x^2h_x$ and, on the other hand, $G_{xx}^h+G_{yy}^h-4G^h_{uv}=-3\left(\partial_x^2 h_x+\partial_y^2 h_y\right)$, so $\partial_x^2h_x=\partial_y^2h_y=0$. Now, we compute $0=2\partial_xG^h_{vx}=h_1\partial_x^2F_0$ and $0=2\partial_yG^h_{vy}=h_1\partial_y^2F_0$. Since $h_1\neq0$, we obtain that  $\partial_x^2F_0=\partial_y^2F_0=0$ and, hence, $R(\partial_x,\partial_v,\partial_v,\partial_x)=R(\partial_y,\partial_v,\partial_v,\partial_y)=0$. Thus, all components of the curvature tensor vanish and the underlying manifold $(M,g)$ is flat, contrary to our assumption.
	
	\subsubsection*{Case 2: $h_1=0$}
	Then, by the weighted Einstein tensor components $G^h_{xx}=\partial_y^2h_y$ and $G^h_{yy}=\partial_x^2h_x$, we find that $h$ takes the form in \eqref{eq:hprwave-1}. Moreover, from $G^h_{uv}= \frac12\tau h$, and since $h>0$, we have $\tau=0$, so $F$ is given by \eqref{eq:Fprwave-1}. For this form of $F$ the Ricci operator is nilpotent.
\end{proof}


\noindent{\it Proof of Theorem~\ref{th:pr-waves-classification}}.
From Lemma~\ref{lemma:$pr$-wave-1}, $h$ and $F$ are given by \eqref{eq:hprwave-1} and \eqref{eq:Fprwave-1}, respectively. A direct computation shows that $|\nh|^2=h_x^2+h_y^2$ so $\nh$ cannot be timelike. Thus, assume first that $(M,g,h)$ is an isotropic solution, so we get that $h_x=h_y=0$ and $h(u,v,x,y)=h(v)$. Now, we see that
\[
2G_{vx}^h=h\, \partial_x F_1, \text{ and } 2G_{vy}^h=h\, \partial_y F_1,
\]
so $\partial_x F_1=\partial_y F_1=0$ and $F_1(v,x,y)=F_1(v)$. But this implies that the only non-vanishing component of the Ricci operator is $\Ric(\partial_v)=-\frac12 \left(\partial_x^2 F_0+\partial_y^2 F_0\right)\partial_u$, so $\Ric$ is $2$-step nilpotent and the $pr$-wave is indeed a $pp$-wave (see \cite{leistner}). Thus, there exist specific coordinates \eqref{eq:local-coord-pr-xeral} with $\partial_uF=0$, where the only non-zero component of $G^h$ is $2G^h(\partial_v,\partial_v)= -2 h'' -h\left(\partial_x^2 F+\partial_y^2 F\right)$, so   $\partial_x^2 F+\partial_y^2 F=-\frac{2h''(v)}{h(v)}=\gamma(v)$. This is a sufficient condition for a $pp$-wave to have harmonic curvature (see \cite{BV-GR-VR}), and Theorem~\ref{th:pr-waves-classification}~(1) follows.	

Assume now that $\nh$ is spacelike. Recall that, by Lemma~\ref{lemma:$pr$-wave-1}, $h$ and $F$ are given by \eqref{eq:hprwave-1} and  \eqref{eq:Fprwave-1}, respectively. Since $|\nh|^2=h_x^2+h_y^2$, we assume without loss of generality that $h_x\neq 0$ (otherwise, interchange the $x$ and $y$ coordinates). Under this condition, we compute the vacuum weighted Einstein equation component $0=G^h_{vx}=-h_x'+\frac{1}2(h_0+xh_x+yh_y)\partial_x F_1$. We solve this PDE to find
\begin{equation}\label{eq:F1alpha}
	F_1(v,x,y)=\alpha(v,y)+2\frac{\log(h_0(v)+x h_x(v)+yh_y(v))h_x'(v)}{h_x(v)}.
\end{equation}
For this form of $F_1$, we compute $0=2\partial_x G_{vy}^h=h_x\partial_y\alpha$ so $\alpha=\alpha(v)$. Moreover, in this case $0=G_{vy}^h=\frac{h_yh_x'}{h_x}-h_y'$. Hence, we have $h_y(v)=Ah_x(v)$ for some $A\in \mathbb{R}$ (this includes the case $h_y=0$). With this, all components of the weighted Einstein equation vanish, except for $G^h_{vv}$, which is given by expression \eqref{eq:remaining-condition-pr-wave}. 
The Ricci operator is given by 
\[
\Ric(\partial_u)=0,\; \Ric(\partial_v)=\star \partial_u+\frac{h_x'}{h} \partial_{x}+ \frac{A h_x'}{h} \partial_{y},\; \Ric(\partial_x)=\frac{h_x'}{h} \partial_u,\; \Ric(\partial_y)=\frac{A h_x'}{h} \partial_u,
\]
where the expression of the coefficient $\star$ is irrelevant for the nilpotency of $\Ric$. Notice that $\Ric$ is $3$-step nilpotent if and only if $h_x'\neq 0$. This case corresponds to Theorem~\ref{th:pr-waves-classification}~(2)(b). 

Now, assume $h_x'=0$. In this case, the Ricci operator becomes $2$-step nilpotent, so the $pr$-wave is indeed a $pp$-wave and we can assume $\partial_uF=0$. Finally, a straightforward calculation shows that the only remaining non-vanishing component of the weighted Einstein equation is given by \eqref{eq:condition-pp-wave}. \qed

\begin{remark}
	Note, from Theorem~\ref{th:pr-waves-classification}~(1), that any isotropic solution on a $pr$-wave is a $pp$-wave with harmonic curvature. Furthermore, any $pp$-wave with harmonic curvature gives rise to an isotropic solution, since there always exists a local solution of the ODE $2h''+h \gamma=0$ for a given $ \gamma(v)=\partial_x^2F+\partial_y^2F$. In particular, if $\gamma$ is constant, the ODE reduces to the harmonic oscillator equation. Thus, for $\gamma<0$, the density function takes the form $h(v)=A e^{\frac{\sqrt{-\gamma } v}{\sqrt{2}}}+B
	e^{-\frac{\sqrt{ -\gamma } v}{\sqrt{2}}}$, so the solution can be extended to all $\mathbb{R}$ for appropriate values of $A$ and $B$. We refer to the next section for geodesically complete solutions. 
\end{remark}

\begin{remark}\rm
	Non-isotropic solutions on $pr$-waves are described in  Theorem~\ref{th:pr-waves-classification}~(2). Those with $2$-step nilpotent Ricci operator are realized on $pp$-waves, but not every $pp$-wave gives rise to a solution, since equation \eqref{eq:condition-pp-wave} imposes restrictions on the function $F$. However, given a (real) analytic density function of the form $h=h_0(v)+(x+Ay)h_x$, there always exist local analytic solutions of \eqref{eq:condition-pp-wave}. Indeed,  \eqref{eq:condition-pp-wave} is a second order quasi-linear PDE that we can write as
	\begin{equation}\label{eq:CK-th1}
		\partial_x^2F=-\frac{1}{h}\left(2h_0''+h_x(\partial_xF+A \partial_yF)\right)- \partial_y^2F.
	\end{equation}
	We consider the non-characteristic hypersurface $x=0$ for this PDE and set analytic initial data $F_{|_{x=0}}=\varphi_0$ and $\partial_xF_{|_{x=0}}=\varphi_1$. Now, the Cauchy-Kovalevskaya Theorem guarantees that there exists an analytic solution $F$ to equation \eqref{eq:CK-th1} (see, for example, \cite{Evans}), thus giving rise to a solution with the prescribed density. 
	
	The situation is similar for $3$-step nilpotent solutions that are realized on $pr$-waves. Although not every $pr$-wave results in a solution of the vacuum weighted Einstein field equations, for a given analytic density function of the form $h(v,x,y)=h_0(v)+(x+Ay)h_x(v)$, there exist forms of $F$ that give rise to $pr$-waves $(M,g)$ so that $(M,g,h)$ is a solution of \eqref{eq:vacuum-Einstein-field-equations}. Indeed, since $h_x\neq0$ and $h(v,x,y)=h_0(v)+h_x(v)(x+Ay)>0$, the hypersurface $x=0$ is non-characteristic for the PDE \eqref{eq:remaining-condition-pr-wave}. Thus, for analytic functions $h_0$, $h_x$, $\alpha$ and analytic boundary conditions, due to the Cauchy-Kovalevskaya Theorem, there is a unique local analytic solution of the corresponding Cauchy problem.
\end{remark}

\section{Solutions with special geometric features}\label{sect:examples}

Our purpose in this section is to use the local description in Section~\ref{sect:pr-waves} to provide explicit examples of solutions that show different properties and behavior. Non-isotropic solutions with harmonic curvature have a distinguished lightlike vector field, indeed they are realized on Kundt spacetimes (see \cite{Brozos-Mojon-noniso}). Therefore, to extend the understanding of this kind of solutions, we begin by focusing on classifying all $pr$-waves which give rise to a solution with this curvature property. Then we turn our attention to providing examples of solutions realized on geodesically complete families of $pr$-waves (such as Cahen-Wallach spacetimes), and explain obstructions to geodesic completeness for non-isotropic solutions.
  
\subsection{$pr$-waves with harmonic curvature}
Solutions to the vacuum weighted Einstein field equations were considered in general in \cite{Brozos-Mojon-noniso} with additional restrictions on the geometry, like being locally conformally flat or having harmonic Weyl tensor. Notice that the latter condition, since the scalar curvature is constant by Lemma~\ref{le:const-sc}, is equivalent to the harmonicity of the curvature tensor. Moreover, by Theorem~\ref{th:pr-waves-classification}~(1), isotropic solutions on $pr$-waves have harmonic curvature. Hence, motivated by these facts, we give an explicit description of the possible solutions on $pr$-waves which have harmonic curvature.

\begin{corollary}\label{co:pr-wave-harmonic}
	Let $(M,g,h)$ be a 4-dimensional, non-flat solution of the vacuum weighted Einstein field equations, realized on a $pr$-wave with harmonic curvature. Then the following holds:
	\begin{enumerate}
		\item If $\nh$ is lightlike, then $(M,g)$ is a $pp$-wave. In coordinates as in \eqref{eq:local-coord-pr-xeral} with $\partial_uF=0$, $F$ satisfies $\partial_x^2F+\partial_y^2F=\gamma(v)$ and the density $h=h(v)$ is subject to the ODE $2h''+h\gamma=0$.
		\item If $\nh$ is spacelike, then $(M,g)$ is a plane wave. Moreover, the metric in \eqref{eq:local-coord-pr-xeral}  takes the form
	\[F(v,x,y)=F_x(v)x^2+F_y(v)y^2-2A(F_x(v)+2F_y(v))xy\]
	with $A\in \mathbb{R}$ and $F_x$, $F_y$ functions subject to the relation $(2-A^2)F_x+(1-2A^2)F_y=0$. The density function has the form $h(u,v,x,y)=h_0(v)+ (x+Ay)h_x$ with $h_0$ obeying  $h_0''(v)+(F_x(v)+F_y(v)) h_0(v)=0$.
	\end{enumerate}
\end{corollary}

\begin{proof}
It was shown in Theorem~\ref{th:pr-waves-classification}~(1) that isotropic solutions of the vacuum weighted Einstein field equations  are realized on manifolds with harmonic curvature, so Corollary~\ref{co:pr-wave-harmonic}~(1) holds and we focus on the case with $\nh$ spacelike.  From Theorem~\ref{th:pr-waves-classification}, the Ricci operator is $2$ or $3$-step nilpotent. If $\Ric$ is $3$-step nilpotent, then $F$ takes the form given in Theorem~\ref{th:pr-waves-classification}~(2)(b). In this case, we compute  \[
0=(\nabla_{\partial_u}\rho)(\partial_v,\partial_v)-(\nabla_{\partial_v}\rho)(\partial_u,\partial_v)=\frac{\left(A^2+1\right) h_x(v) h_x'(v)}{
	h^2}.
\] 
Hence $h_x'=0$ and we conclude that the Ricci operator is $2$-step nilpotent. Since the image of $\Ric$ is totally isotropic, the underlying manifold is a $pp$-wave (see \cite{leistner}).

Thus, adopting the notation in Theorem~\ref{th:pr-waves-classification}~(2)(a), we  take coordinates $(u,v,x,y)$ as in \eqref{eq:local-coord-pr-xeral} with $\partial_u F=0$, such that the density function has the form $h(u,v,x,y)=h_0(v)+ (x+Ay)h_x$ with $h_x\neq 0$. Recall that, in these coordinates, the curvature is harmonic if and only if $\partial_x^2 F+\partial_y^2 F=\beta(v)$, for an arbitrary function $\beta\neq 0$.

Moreover,  
the only non-vanishing term of the weighted Einstein tensor is given by \eqref{eq:condition-pp-wave}. Using the fact that $\partial_x^2 F+\partial_y^2 F=\beta(v)$, the crucial equation becomes $2h_0''+h_x(\partial_xF+A \partial_yF)+h\beta=0$. Differentiating with respect to $x$ and $y$ yields
\begin{equation}\label{eq:dereqs-ppwave}
	0=h_x(A\partial_x\partial_yF+(2\partial_x^2F+\partial_y^2F)), \quad 0=h_x(\partial_y\partial_xF+A(2\partial_y^2F+\partial_x^2F)).
\end{equation}
Since $h_x\neq 0$, we get that $A\partial_x\partial_yF+(2\partial_x^2F+\partial_y^2F)=0$ and that $\partial_y\partial_xF+A(2\partial_y^2F+\partial_x^2F)=0$. If $A=0$, then $\partial_y\partial_xF=0$, while if $A\neq0$, combining both equations yields $(1+A^2)\partial_x\partial_yF+3A \beta=0$, and hence $\partial_x\partial_yF=-\frac{3A\beta(v)}{1+A^2}$. Thus, for any value of $A$, we have $\partial_x^2\partial_yF=\partial_x\partial_y^2F=0$. Moreover, differentiating $A\partial_x\partial_yF+(2\partial_x^2F+\partial_y^2F)=0$ with respect to $x$ and $y$, we get that $\partial_x^3 F=\partial_y^3F=0$ too.

It follows that $F$ is a polynomial of order two in the variables $x,y$, whose coefficients are smooth functions of $v$. Hence, the underlying manifold is a plane wave, and can be further normalized so that 
\[
	F(v,x,y)=F_y(v)y^2+F_x(v)x^2+F_{xy}(v)xy
\]
for some smooth functions $F_y$, $F_x$ and $F_{xy}$. With this, from \eqref{eq:dereqs-ppwave} we get
\[
		0=AF_{xy}+2(2F_x+F_y), \quad 0=F_{xy}+2A(2F_y+F_x).
\]
We can solve the second equation above to find $F_{xy}=-2A(2F_y+F_x)$, and substituting this into the first one yields $(2-A^2)F_x+(1-2A^2)F_y=0$. 

Now, equation~\eqref{eq:condition-pp-wave} reduces to \[
0=h_0''(v)+(F_x+F_y) h_0(v)-h_x x \left(\left(A^2-2\right)
F_x+\left(2 A^2-1\right) F_y\right),
\]
and, since $(2-A^2)F_x+(1-2A^2)F_y=0$, we have that $h_0$ satisfies $h_0''(v)+(F_x+F_y) h_0(v)=0$. This completes the proof of Corollary~\ref{co:pr-wave-harmonic}~(2).
\end{proof}
%

\subsection{Geodesically complete solutions}\label{subsect:geosically-complete}

The family of plane waves appears in Corollary~\ref{co:pr-wave-harmonic}~(2), since every $pr$-wave with harmonic curvature that results in a non-isotropic solution for equation \eqref{eq:vacuum-Einstein-field-equations} is indeed a plane wave. Inspired by this fact and taking into account that plane waves are geodesically complete (see \cite{candela-flores-sanchez}), one may look for global solutions in this context. However, one of the difficulties in finding global solutions of the vacuum weighted Einstein field equations is that some geodesically complete spacetimes do not admit a globally defined density function. Indeed, notice that non-isotropic solutions described in Corollary~\ref{co:pr-wave-harmonic}~(2) cannot be extended for any $(x,y)\in \mathbb{R}^2$. In this section we illustrate this fact and provide some global examples.

{\it Cahen-Wallach spaces}. 
Cahen-Wallach spaces are the only indecomposable but not irreducible symmetric manifolds (see \cite{Cahen-Leroy-Parker-Tricerri-Vanhecke,cahen-wallach}). They are plane waves, hence geodesically complete, and can be written in coordinates $(u,v,x,y)$ on $\mathbb{R}^4$ as in \eqref{eq:local-coord-pr-xeral} with $F(v,x,y)= a x^2+by^2$. Since Cahen-Wallach spaces are symmetric, they have harmonic curvature and we can apply Corollary~\ref{co:pr-wave-harmonic} directly to obtain the following families of solutions: 

\begin{itemize}
	\item {\it Global isotropic solutions.} The density function $h=h(v)$ satisfies the equation $0=h''+(a+b)h$. Thus, only spacetimes with $a+b<0$ result in vacuum global solutions for an appropriate density function. Indeed, $a+b\geq0$ yields densities which turn nonpositive for certain values of $v$. Moreover if $a+b<0$, $h$ has the form $h(v)=c_1 e^{v \sqrt{-a-b}}+c_2 e^{-v\sqrt{-a-b}}$, $c_1,c_2\in \mathbb{R}$. To define a global solution we can take  $c_1,c_2\geq 0$, maybe with $c_1=0$ or $c_2=0$, so as to keep $h>0$ for all $v\in \mathbb{R}$.
	\item {\it Local non-isotropic solutions.} In virtue of Corollary~\ref{co:pr-wave-harmonic}~(2), the form of $F$ is restricted to $F(v,x,y)=a(x^2-2y^2)$ and the density function takes the form $h(v,x)=h_0(v)+h_x x $ with $h_x\neq 0$ and $0=h''_0-ah_0$. Thus, for a fixed value of $v$, the density turns non-positive for large enough values of $x$. Hence, only local solutions are admissible.
\end{itemize}

{\it A family of manifolds with recurrent curvature.}  Consider a plane wave on $\mathbb{R}^4$, given by the metric \eqref{eq:local-coord-pr-xeral} with
\begin{equation}\label{eq:egorov-metric}
	F(v,x,y)= f(v) (ax^2+by^2),
\end{equation}
in the usual coordinates $(u,v,x,y)$, for a certain non-constant function $f(v)$. One characteristic property of these spacetimes is that they have recurrent curvature, this is, $\nabla R=\sigma \otimes R$ for a $1$-form $\sigma$, but they are not locally symmetric (see \cite{Galaev,walker}).  Whenever $a=b$, the resulting manifolds are locally conformally flat and they are referred to as {\it Egorov spaces}. Note that Egorov spaces such that $f(v)$ is a constant are Cahen-Wallach manifolds belonging to the family of {\it $\varepsilon$-spaces}. Egorov spaces and $\varepsilon$-spaces are notable examples of Lorentzian manifolds with large isometry groups (see \cite{calvaruso-garcia}). 

Spacetimes given by \eqref{eq:egorov-metric} are not homogeneous in general, but they have harmonic curvature tensor, so we can apply Corollary~\ref{co:pr-wave-harmonic} to find the following families of solutions:
\begin{enumerate}
	\item {\it Global isotropic solutions.} For a density function of the form $h=h(v)$ we have that $\nh$ is lightlike. Moreover, the only non-vanishing component of \eqref{eq:vacuum-Einstein-field-equations} becomes (see Corollary~\ref{co:pr-wave-harmonic}~(1))
	\begin{equation}\label{eq:compo-non-cero}
		0=G^h_{vv}=-h''-(a+b) f h.
	\end{equation}
hence, we can choose appropriate values of $f$ that provide global solutions. For example, we consider the following:
	\begin{enumerate}
		\item Let $h=h(v)$ and set $f(v)=\frac{1}{h(v)}$. Then \eqref{eq:compo-non-cero} reduces to
\[
	0=-G_{vv}=(a+b)+h''(v).
\]
Thus, $h(v)=-\frac{a+b}2v^2+c_1v+c_2$ is the general solution of the ODE. Choosing $a+b<0$ and $c_2>-\frac{c_1^2}{2 (a+ b)}$, results in $h(v)>0$ for all $v\in \mathbb{R}$, giving rise to a global solution.

		\item Let $h=h(v)$ and $f(v)=-\frac{1+(a+b) e^{4 v}}{a+b}$. Now, the ODE \eqref{eq:compo-non-cero} is 
		\[
		h(v) \left(1+(a +b) e^{4 v}\right)-h''(v)=0.
		\] 
The general solution of this ODE for $a+b>0$ is
\[
h(v)=e^{-v}\left(c_1 \cosh \left(\frac{1}{2} e^{2
		v} \sqrt{a+b}\right)+c_2 \sinh \left(\frac{1}{2} e^{2 v} \sqrt{a+b}\right)\right).
\]
Thus, for $a+b>0$, we obtain geodesically complete solutions of \eqref{eq:vacuum-Einstein-field-equations} by taking $c_1> |c_2|$.
\end{enumerate}
	\item {\it Local non-isotropic solutions.} By Corollary~\ref{co:pr-wave-harmonic}~(2), we have $F(v,x,y)=af(v)(x^2-2y^2)$ and the density function takes the form $h(v,x)=h_0(v)+ h_x x$, with $h_x\neq 0$ and $h_0''(v)-af(v) h_0(v)=0$. As for Cahen-Wallach spaces, for a fixed $v$, the density turns non-positive for large enough values of $x$, so only local solutions are admissible. For example, based on the form of solutions given in the isotropic case, we have:
	\begin{enumerate}
		\item Let $f(v)=\frac{1}{h_0(v)}$ to see that $F(v,x,y)=a \frac{x^2-2y^2}{\frac{a}2 v^2+c_1v+c_2}$ and $h(v,x)=\frac{a}2v^2+c_1v+c_2+ h_x x$ define a local non-isotropic solution for $a\neq 0$ and $h_x\neq 0$.
		\item Let $f(v)=\frac{1}a+e^{4v}$ for $a>0$ to obtain that $F(v,x,y)=\left(1+a e^{4v}\right) (x^2-2 y^2)$ and $h(v)=e^{-v} \left(c_1 \cosh \left(\frac{1}{2}
			\sqrt{a} e^{2 v}\right)+c_2 \sinh \left(\frac{1}{2} \sqrt{a} e^{2 v}\right)\right)+x h_x$ define non-isotropic solutions for positive values of $x$ if $h_x>0$.		  
	\end{enumerate}
\end{enumerate}

	

\section{Conclusions}

We have considered a broad family of Lorentzian metrics, that of $pr$-waves, which includes many examples of solutions to equation \eqref{eq:vacuum-Einstein-field-equations}, and appears naturally in previous studies of this equation. The analysis we have carried out and the given examples illustrate some important aspects:
\begin{itemize}
	\item Although the geometry of isotropic solutions is very rigid, as pointed out in \cite{Brozos-Mojon,Brozos-Mojon-noniso}, under some geometric conditions non-isotropic solutions are also very rigid. For example, in the context of Corollary~\ref{co:pr-wave-harmonic}, the structure is more restricted for non-isotropic than for isotropic solutions.
	\item Although some geodesically complete examples are given, the extension of the density function to the whole manifold becomes an important issue. Indeed, the analysis in Subsection~\ref{subsect:geosically-complete} illustrated the difficulty of finding global solutions in certain spaces, especially in the non-isotropic case.
	\item The condition of harmonic curvature considered in Corollary~\ref{co:pr-wave-harmonic} reduces the admissible solutions to $pp$-waves (for isotropic solutions) and plane waves (for non-isotropic ones). Note that this harmonicity condition does not increase the rigidity of isotropic solutions with respect to Theorem~\ref{th:pr-waves-classification}. However, in the non-isotropic case, removing the condition allows for local solutions with 3-step nilpotent Ricci operator, which are $pr$-waves but not $pp$-waves. This contrasts with the broader family of Kundt spacetimes, which do allow examples with harmonic curvature and $3$-step nilpotent Ricci operator (see \cite{Brozos-Mojon-noniso}).
\end{itemize}

\end{document}